\documentclass{article}
\usepackage[utf8]{inputenc}
\usepackage{fullpage}

\usepackage{wrapfig}
\usepackage{graphicx}
\usepackage{mathtools,amsmath,amssymb,amsfonts,amsthm}
\usepackage{adjustbox}
\usepackage{microtype}
\usepackage[authoryear]{natbib}
\usepackage{lipsum}
\usepackage{tikz,tikz-cd}
\usetikzlibrary{arrows}
\usetikzlibrary{decorations.markings}
\tikzset{degil/.style={line width=0.5pt,double distance=5pt,
        decoration={markings,
        mark= at position 0.5 with {
              \node[transform shape] (tempnode) {$\backslash\backslash$};
              }
          },
          postaction={decorate}
}
}

\tikzset{
commutative diagrams/.cd,
arrow style=tikz,
diagrams={>=open triangle 45, line width=0.5pt}}

\usepackage{tikz}
\usetikzlibrary{automata,positioning, arrows}

\usetikzlibrary{calc}

\usepackage{color}
\usepackage{floatrow}   
\usepackage{url}

\usepackage{enumitem}

\newcommand{\myDots}{\hbox to 0.9em{.\hss.\hss.}}

\graphicspath{{./figures/}}

\usepackage{tabularx}

\usepackage{xcolor}
\usepackage{ragged2e}
\usepackage{verbatim}
\usepackage{xspace}

\usepackage{setspace}

\usepackage{colortbl}

\usepackage{ifthen}
\usepackage{subcaption}

\definecolor{green1}{rgb}{0.0, 0.5, 0.0}

\usepackage[subpreambles=false]{standalone}

\theoremstyle{plain}
\newtheorem{thm}{Theorem}
\newtheorem{lemma}{Lemma}
\newtheorem{prop}{Proposition}
\newtheorem{cor}{Corollary}

\theoremstyle{definition}
\newtheorem{fact}{Fact}
\newtheorem{defn}{Definition}

\theoremstyle{remark}

\newtheorem{rem}{Remark}

\newcommand{\R}{\mathbb{R}}{}
\newcommand{\N}{\mathbb{N}}

\newcommand{\bS}{\mathbb{S}}

\newcommand{\B}{\mathbb{B}}
\newcommand{\A}{\mathbb{A}}

\newcommand{\cA}{\mathcal{A}}

\newcommand{\cC}{\mathcal{C}}

\newcommand{\cK}{\mathcal{K}}

\newcommand{\cG}{\mathcal{G}}
\newcommand{\M}{\langle M \rangle}

\newcommand{\wi}{\hat {w}}
\newcommand{\wj}{\hat {v}}

\newcommand{\cF}{\mathcal{F}}

\newcommand{\cS}{\mathcal{S}}

\newcommand{\cW}{\mathcal{W}}

\newcommand{\wt}{\widetilde}

\newcommand{\toS}{\rightrightarrows}

\newcommand{\re}{{\rm rem}}
\newcommand{\eS}{{\bf \epsilon}}
\newcommand{\cut}{{\rm cut}}

\DeclareMathOperator*{\argmin}{argmin}

\DeclareMathOperator*{\mspan}{span}

\makeatletter
\newsavebox\myboxA
\newsavebox\myboxB
\newlength\mylenA

\newcommand*\pbar[1]{%
  \hbox{%
     \vbox{%
      \hrule height 0.7pt % The actual bar
      \kern0.35ex%         % Distance between bar and symbol
      \hbox{%
         \kern-0.0em%      % Shortening on the left side
         \ensuremath{#1}%
         \kern-0.0em%      % Shortening on the right side
      }%
     }%
  }%
} 

\newcommand*\xbar[2][0.75]{%
    \sbox{\myboxA}{$\m@th#2$}%
    \setbox\myboxB\null% Phantom box
    \ht\myboxB=\ht\myboxA%
    \dp\myboxB=\dp\myboxA%
    \wd\myboxB=#1\wd\myboxA% Scale phantom
    \sbox\myboxB{$\m@th\pbar{\copy\myboxB}$}%  Overlined phantom
    \setlength\mylenA{\the\wd\myboxA}%   calc width diff
    \addtolength\mylenA{-\the\wd\myboxB}%
    \ifdim\wd\myboxB<\wd\myboxA%
       \rlap{\hskip 0.5\mylenA\usebox\myboxB}{\usebox\myboxA}%
    \else
        \hskip -0.5\mylenA\rlap{\usebox\myboxA}{\hskip 0.5\mylenA\usebox\myboxB}%
    \fi}
\makeatother

\allowdisplaybreaks
\title{Feedback stabilization of switched systems under arbitrary switching:\\A convex characterization}

\author{Thiago Alves Lima \thanks{All authors contributed equally to this work. \\T.A.L. is with the Aeronautics Institute of Technology (ITA), Fortaleza, 60415-513, CE, Brazil. M.D.R is with the Department of Electronics and Telecommunications, Politecnico di Torino, Corso Duca degli Abruzzi, 24, 10129 Torino, Italy.  A.G. is with the Université Paris-Saclay, CNRS, CentraleSupélec, Laboratoire des Signaux et Systèmes, 91190, Gif-sur-Yvette, France.\\
        \textit{Email addresses:} \textbf{\scriptsize thiago.lima@gp.ita.br} (T. Alves Lima), \textbf{\scriptsize matteo.dellarossa@polito.it} (M. Della Rossa), \textbf{\scriptsize antoine.girard@centralesupelec.fr} (A. Girard).} \and Matteo Della Rossa \and Antoine Girard}

\begin{document}

\maketitle    

 \begin{abstract}
In this paper, we study stabilizability of discrete-time switched linear systems where the switching signal is considered as an arbitrary external input (and not a control variable). We characterize feedback stabilization via \textcolor{black}{a hierarchy} of necessary and sufficient linear matrix inequalities (LMIs) conditions based on novel graph structures. We analyze both the cases in which the controller has (or has not) access to the current switching mode, the so-called mode-dependent and mode-independent settings, providing specular results. Moreover, our approach provides 
explicit piecewise-linear and memory-dependent linear controllers, 
highlighting the connections with existing stabilization approaches. The effectiveness of the proposed technique is finally illustrated with the help of some numerical examples.
 \end{abstract}

\section{Introduction}

 In this paper, given a \textcolor{black}{natural number $M>0$}  and pairs  $(A_i,B_i)\in \R^{n\times n}\times \R^{n\times m}$ with $ i\in \{1,\dots,M\}$, we study the \emph{(discrete-time) control switched linear system} defined~by
\begin{equation}\label{eq:SwitchedSystemInputIntro}
x(k+1)=A_{\sigma(k)}x(k)+B_{\sigma(k)}u(k),\;\;\;k\in \N,
\end{equation}
\textcolor{black}{where $\N$ denotes the set of natural numbers including~$0$.}
For a general introduction to switched dynamical systems, we refer to~\cite{Lib03} and references therein.
In equation~\eqref{eq:SwitchedSystemInputIntro}, $x:\N\to \R^n$, constrained to an initial conditions of the form $x(0)=x_0\in \R^n$ is the (unique) internal \emph{state}. Then, system~\eqref{eq:SwitchedSystemInputIntro} can be interpreted as a model with two different kinds of external inputs: a classical input  $u:\N\to \R^m$ and a switching-signal input $\sigma:\N\to \{1,\dots, M\}$ which takes values in a discrete set. Both these inputs can be considered, depending on the modeling requirements, as external \emph{noises/disturbances}, unmodifiable by the user, or as external \emph{control inputs} that the user might choose in order to achieve a prescribed task (e.g., reachability, minimization of a cost functional, etc.). 

A classical question in this context is the \emph{stabilization problem}, which  requires to establish the  existence and, when possible, the construction, of inputs $u$ and/or $\sigma$ driving any initial state to $0$. For the case in which $u\equiv 0$ and thus the only control input is the signal $\sigma:\N \to \{1,\dots, M\}$, a mature literature is already available, see~\cite{FiaGir16,FiaJun14,JunMas17,DettJun20}. For the case in which \emph{both} $u:\N\to \R^m$ and $\sigma:\N\to \{1,\dots, M\}$ are designed to achieve stabilization, we \textcolor{black}{refer to~\cite{LinAnt08,LinAnt09,ZhangAbate09,FiaTar17}} in which sufficient conditions are provided, \textcolor{black}{relying} on the design of ``joint''-control Lyapunov functions.

\textcolor{black}{In this paper, instead, we consider the case in which the switching signal is arbitrary and unmodifiable by the user}. The input $u:\N\to \R^m$ is then the unique control input to be designed to achieve  stability of~\eqref{eq:SwitchedSystemInputIntro}, \emph{uniformly} over the class of switching signals.
This problem has been tackled in a closely related setting, namely for \emph{linear parameter-varying} (LPV) systems, which can be considered as a ``convexified'' version of~\eqref{eq:SwitchedSystemInputIntro}; for a general introduction see~\cite{BlaMia03}. Stabilization results in this framework have been provided in~\cite{BlaMia03,GoeHu06,BlaMia2008,HuBla10}, notably establishing that the use of \emph{linear} feedback controllers is generally restrictive and thus proposing the design of \emph{piecewise linear} feedback controllers. These results apply, under some hypothesis, also to the switched systems as in~\eqref{eq:SwitchedSystemInputIntro}.
Indeed, substantially confirming the results obtained in the LPV setting while leading to novel (nonlinear) feedback design techniques, the same stabilization problem has been studied, in a genuinely switched systems setting, in~\cite{HuShen17,LeeDull06,LeeKha09}.

In this paper, we propose classes of controllers in feedback form in two different frameworks: firstly, considering feedback maps depending, at each time-instant $k\in \N$, only on the current system state $x(k)$ (\emph{mode-independent feedback}), and secondly considering control schemes depending on the current state $x(k)$ and on the current switching mode $\sigma(k)$ (\emph{mode-dependent feedback}).
In both cases, we introduce and study  notions of \emph{feedback stabilizability} for~\eqref{eq:SwitchedSystemInputIntro}.
Since in the studied setting the switching signal $\sigma$ is arbitrary, we make use of a \emph{difference inclusion} representation of~\eqref{eq:SwitchedSystemInputIntro},  allowing us to provide a complete characterization of the studied notions of stabilizability. We prove an equivalence result, informally stated here:
\vspace{-0.1cm}
\begin{center}
\justify{\emph{System~\eqref{eq:SwitchedSystemInputIntro} is (robustly) feedback stabilizable if and only if there exist \emph{piecewise} linear feedback controllers (and associated \emph{piecewise} quadratic Lyapunov function) obtained via \textcolor{black}{an asymptotically tight hierarchy of} \emph{graph-based} linear matrix inequalities (LMIs).}}\\
\end{center}
\vspace{-0.1cm}
\textcolor{black}{The framework of graph-based stabilization design leading to piecewise linear control laws has been recently proposed in~\cite{DelRosAlv24}; however sufficient-only LMI conditions were presented.} The strategy was developed by leveraging the connections between path-complete Lyapunov functions (see~\cite{AhmJun:14}) and min-max of quadratics Lyapunov functions, which had been previously explored only in the context of \emph{stability analysis}~\citep{PhiAthAng}.
\textcolor{black}{Although} the sufficient conditions obtained in~\cite{DelRosAlv24} (based on the so-called De Bruijn graphs) were conjectured to also be \emph{necessary}, this  remains unproven. \textcolor{black}{Motivated by this fundamental question,} in this manuscript, we propose a different hierarchy of LMI conditions based on a novel class of graphs, called \emph{feedback trees}. We prove that the satisfiability of this new class of conditions is \emph{both sufficient and necessary} for the studied notions of stabilizability, \textcolor{black}{thus providing, for the first time, an equivalence (i.e., a complete characterization) in terms of LMIs, in contrast with existing results in the literature. For instance, the results in~\cite{BlaMia03,BlaMia2008,HuBla10} rely either on polyhedral Lyapunov functions and/or on BMI formulations, in contrast to LMIs which can be solved in polynomial time.}

A central instrumental result in proving such equivalence is the fact that feedback stabilizability of~\eqref{eq:SwitchedSystemInputIntro} is structurally \emph{robust}, in a sense we will clarify. Such achievements are obtained with formal and technical tools not used in~\cite{DelRosAlv24}, as converse Lyapunov functions theory for difference inclusions~\citep{KELLETT2004395,KellTeel} and homogeneity. %~\citep{Poly20}.

\textcolor{black}{
While establishing the necessity of the newly proposed conditions, we also obtain a significant intermediate result: these conditions guarantee stabilizability of~\eqref{eq:SwitchedSystemInput} by \emph{linear memory-dependent controllers}, i.e., linear controllers that may utilize the entire history of the switching signal. In this context, the external switching signal remains arbitrary. However, in the \emph{mode-independent case}, the controller can store and leverage past values of the signal for stabilization, whereas in the \emph{mode-dependent case}, both the \emph{current and past values} can be utilized.} \textcolor{black}{Such memory-dependent \emph{linear} controllers are explicitly provided by the proposed numerical conditions, allowing for an alternative feedback implementation.  
In this construction, we make partial use of results concerning dynamical systems depending on a notion of language-theory memory, as introduced in~\citep{DelRosJung24}.}
Moreover, this allows us to clarify the relations between our results and the ones concerning memory-dependent linear controllers, as for example the ones provided in~\citep{BLIMAN2003325,LeeKha09,EssLee14} and references therein.
This final analysis allows us to provide an almost complete picture of the connections between stabilizability notions and techniques for~\eqref{eq:SwitchedSystemInputIntro}.

\noindent \textbf{Notation:} We denote by $\N$ (resp. $\N_+$) the set of natural numbers including (resp. excluding) $\{0\}$. Given $M\in \N_+$, we define $\M:=\{1,\dots, M\}$. Given a finite set $\Sigma$ (the \emph{alphabet)} and $T\in \N$, we denote by $\Sigma^T$ the set of strings of length $T$ in $\Sigma$, and we denote its elements by $\wi=(i_0,\dots, i_{T-1})\in \Sigma^T$.
 By convention, $\Sigma^0=\{ \bf \epsilon\}$ where ``$\bf \epsilon$'' is an auxiliary symbol, which stands for ``empty string''. 
 With $\Sigma^\star:=\bigcup_{T\in \N}\Sigma^T$ we denote the \emph{Kleene closure} of $\Sigma$, that is, the set of all finite-length strings in $\Sigma$.
 With $\Sigma^\omega$ we denote the \emph{$\omega$-closure} of $\Sigma$, i.e. the set of all the (infinite) sequences in $\Sigma$. Given any $\sigma \in \Sigma^\omega$ and any $a\leq b\in \N$ we denote by $[\sigma]_{[a,b]}\in \Sigma^{b-a+1}$ the \emph{restriction of $\sigma$ to the interval $[a,b]$}, i.e. $[\sigma]_{[a,b]}=(\sigma(a),\sigma(a+1),\dots, \sigma(b))\in \Sigma^{b-a+1}$. Given $x\in\R^n$, with $|x|$ we denote its Euclidean norm,  with $\B$ we denote the  closed unit ball, i.e., $\B:=\{x\in \R^n\;\vert\;|x|\leq 1\}$.
 We denote by $\bS^{n\times n}$ the set of the $n\times n$ symmetric matrices, and by  $\bS_+^{n\times n}$  the set of $n\times n$ positive definite matrices. A function $U:\R^n\to \R$ is \emph{positive definite} if $U(0)=0$ and $U(x)>0$ for all $x\neq 0$.
Given $p\in \N$, a function $\Psi:\R^n\to \R^m$ is  \emph{homogeneous of degree $p$} if
$\Psi(\lambda x)=\lambda^p \Psi(x)$, for all $\lambda\in \R$ and all $x\in \R^n$.
It is  \emph{absolutely homogeneous of degree 
p} if $\Psi(\lambda x)=|\lambda|^p \Psi(x)$, for all $\lambda\in \R$ and all $x\in \R^n$.
\section{Preliminaries}\label{sec:Prel}
In this section we introduce some preliminary definitions and results that will be used in what follows.
\subsection{Difference inclusions and robust stability notions}\label{sec:SetValuedStability}
We recall the main notation and definitions concerning \emph{difference inclusions}. For a 
self-contained introduction to set-valued analysis, we refer, for example, to~\cite{RockWets09}.
The notation $F:\R^n\rightrightarrows\R^n$ stands for the fact that $F$ is a \emph{set-valued map} that sends points of $\R^n$ to \emph{subsets} of $\R^n$.
 It is said to be \emph{homogeneous of degree $1$} if, for any $\lambda\in \R$ and any $x\in \R^n$, it holds that $F(\lambda x)=\lambda F(x)$.
Given $A\subseteq \R^n$, we write $F(A):=\bigcup_{x\in A}F(x)$.
Given $F:\R^n\rightrightarrows \R^n$, we consider the \emph{difference inclusion}
\begin{equation}\label{eq:DifferenceInclusion}
x(k+1)\in F(x(k)).
\end{equation}
If $F$ is non-empty valued, for any initial condition $x(0)=x_0\in \R^n$, there exist  \emph{global solutions} to~\eqref{eq:DifferenceInclusion}, i.e.,  functions $\phi:\N\to \R^n$ such that $\phi(0)=x_0$ and $\phi(k+1)\in F(\phi(k))$ for all $k \in \N$.
Given $x_0\in \R^n$ the set of global solutions to~\eqref{eq:DifferenceInclusion} with initial condition $x_0\in \R^n$ is denoted by $\cS_{F,x_0}$. 

\begin{defn}\label{defn:WeakStability}
The difference inclusion in~\eqref{eq:DifferenceInclusion} is said to be \emph{uniformly exponentially stable} (UES) if there exists $C\geq 0$ and $\gamma\in [0,1)$ such that $|\phi(k)|\leq C\gamma^k |x_0|$, for all $x_0\in \R^n$, for all $\phi\in \cS_{F,x_0}$ and for all $k\in \N$.
\end{defn}
The stability notion introduced in~Definition~\ref{defn:WeakStability} is too ``weak'' or ``delicate'' for certain purposes. We then strengthen it, to obtain a \emph{robust} version of it; to do so, we introduce a class of ``\emph{enlargements}'' for the map~$F$. 
%Without loss of generality, we explicitly provide this construction only for the homogeneous-of-degree $1$ case that is considered in this paper; the general case can be trivially obtained, \emph{mutatis mutandis}.
\begin{defn}[\cite{KELLETT2004395}, Eq. (1)]\label{defn:perturbation}
Given $F:\R^n \toS \R^n$ and a continuous  $\rho:\R^n\to \R_+$, the \emph{$\rho$-perturbation of $F$} is given by $F_\rho:\R^n\toS \R^n$ defined by
\[
F_\rho(x):=F(x+\rho(x)\B)+\rho(x)\B.
\]
\end{defn}
In order to avoid misunderstanding, we note that this definition of perturbation, borrowed from~\cite{KELLETT2004395} differs from the one in~\cite[Equation~(2.2)]{KellTeel}. In these references, under some tailored assumptions on $F$, it is implicitly proven that these two notions perturbations lead to the same concept of robustness, roughly speaking. We can now state the following robust notion of stability.
\begin{defn}\label{defn:robustStability}
The inclusion~\eqref{eq:DifferenceInclusion} is  \emph{robustly uniformly exponentially  stable} (RUES) if there exists a continuous $\rho:\R^n\to \R_+$ such that $\rho(x)>0$ for all $x\neq 0$ for which the inclusion
$x(k+1) \in F_\rho(x(k))$ is UES.
\end{defn}

We recall the characterization of (robust) stability via Lyapunov functions; we first need the following definition.

\begin{defn}[(Exponential) Lyapunov functions]\label{defn:LyapunovExponential}
A function $V:\R^n\to \R$ is said to be an \emph{exponential Lyapunov function} for the inclusion~\eqref{eq:DifferenceInclusion} is there exist $a_1,a_2\in \R_{>0}$ and $\gamma\in [0,1)$ such that
\[
\begin{aligned}
a_1|x|\leq & V(x)\leq a_2|x|,\;\;\;\forall \,x\in \R^n,\\
V(f)\leq\gamma &V(x),\;\;\;\forall x\in \R^n,\;\;\forall \,f\in F(x). 
\end{aligned}
\]
\end{defn}
We now recall a pivotal result relating robust stability and existence of continuous Lyapunov functions.
\begin{lemma}\label{lemma:technicalLemma}
Suppose $F:\R^n \toS \R^n$ is non-empty valued.
If there exists a \emph{continuous} exponential Lyapunov function for~\eqref{eq:DifferenceInclusion}, then the inclusion is RUES.
\end{lemma}

This is the ``exponential version'' of the result proven in~\cite[Section 6.1]{KELLETT2004395}, since the Lyapunov functions as in Definition~\ref{defn:LyapunovExponential} satisfy \emph{linear} lower and upper bounds and a \emph{linear} decrease condition. The (straightforward) proof is thus not explicitly reported. Theorem 10 in~\cite{KELLETT2004395} actually shows the equivalence between RUES and the existence of a continuous exponential Lyapunov function under some additional regularity assumptions on $F$ (compact-values and upper semicontinuity). Note that in this manuscript, we only use the necessity part of this theorem, stated above in Lemma 1, which holds for arbitrary set-valued maps $F$ as explicitly outlined in~\cite[Section 6.1]{KELLETT2004395}.

\begin{comment}
 Note that the set-valued map $F$ in Lemma~\ref{lemma:technicalLemma} is not assumed to be compact-valued (as in~\cite{KELLETT2004395}); if Lyapunov inequalities as in Definition~\ref{defn:LyapunovExponential} hold, then by continuity of $V$ they also hold for all $f\in \overline{F(x)}$, for all $x\in \R^n$. Boundedness of $F(x)$ is also straightforward under the hypotheses of Lemma~\ref{lemma:technicalLemma}, since $|f|\leq a_1^{-1}V(f)\leq \gamma a_1^{-1}V(x)\leq \gamma a_1^{-1}a_2|x|$, for all $x\in \R^n$, for all $f\in F(x)$.

\end{comment}

\section{Discrete-time control switched systems}\label{sec:Prel2}
In this section we provide the definition of the considered class of systems and the
studied stabilization notions.
\subsection{Definitions and robust stabilizability notions}
Given $M\in \N_+$ and a set  $\cF=\{(A_i,B_i)\in \R^{n\times n}\times \R^{n\times m}\;\vert\;i\in \M\}$, we study the \emph{discrete-time control switched system} defined by
\begin{equation}\label{eq:SwitchedSystemInput}
x(k+1)=A_{\sigma(k)}x(k)+B_{\sigma(k)}u(k),\;\;\;k\in \N,
\end{equation}
where $\sigma:\N\to \M$ (equivalently, $\sigma\in \M^\omega$) is the \emph{switched signal} and $u:\N\to \R^m$ is a control input.

\textcolor{black}{While $u:\N\to \R^m$ is supposed to be a control input given/chosen by the user, the switching signal $\sigma:\N\to \M$ is considered to be \emph{arbitrary} and \emph{unmodifiable} by the user. 
We now introduce the considered notions of stability for~\eqref{eq:SwitchedSystemInput}.}

\begin{defn}\label{defn:StabNot}
System~\eqref{eq:SwitchedSystemInput} is said to be
\begin{enumerate}[leftmargin=*]
\item \emph{Mode-independent feedback stabilizable} (IFS), 
if there exists  $\Phi:\R^n \to \R^m$ such that, considering $F^{\Phi}:\R^n\toS \R^n$ defined by
$F^\Phi(x)=\{A_ix+B_i\Phi(x)\;\vert\;\;i\in \M\}$,
the closed-loop difference inclusion
\begin{equation}\label{eq:modeIndSys}
x(k+1)\in F^{\Phi}(x(k)) 
\end{equation}
is UES.
\item \emph{Mode-dependent feedback stabilizable} (DFS), 
if there exist  $\Phi_1,\dots, \Phi_M:\R^n \to \R^m$ such that, considering $F_d^{\Phi}:\R^n\toS \R^n$ defined by $F_d^\Phi(x)=\{A_ix+B_i\Phi_i(x)\;\vert\;\;i\in \M\}$,
the closed-loop difference inclusion
\begin{equation}\label{eq:modeDepSys}
x(k+1)\in F_d^{\Phi}(x(k)) 
\end{equation}
is UES.
\end{enumerate}
\end{defn}
Intuitively, system~\eqref{eq:SwitchedSystemInput} is  \emph{mode-independent} feedback stabilizable if there exists a \emph{common} feedback law which stabilizes the system, regardless the underlying switching sequence; the user/controller  does not need to observe the active mode of the switching sequence. On the other hand, the system is \emph{mode-dependent} feedback stabilizable if there exist $M$ feedback maps, one for each mode; in this setting the user/controller needs to observe the current active mode, in order to apply the correct feedback law.
\textcolor{black}{
It is thus clear that IFS implies DFS. The other implication is not true. For an example of a system that is DFS but not IFS we refer to~\cite[Example 5.1]{BlaMiaSav07}.
We now recall a result, stating   that the concepts of stabilizability introduced in~Definition~\ref{defn:StabNot} inherently embodies a form of implicit \emph{robustness} and that, without loss of generality, one can consider \emph{homogeneous of degree $1$} controllers.}
\begin{thm}
\label{thm:HomogenenousCharachterization}
Given $M\in \N_+$ and a set $\cF=\{(A_i,B_i)\in \R^{n\times n}\times \R^{n\times m}\;\vert\;i\in \M\}$, system~\eqref{eq:SwitchedSystemInput} is:
\begin{itemize}
    \item[(A)] 
    IFS if and only if there exists 
 a homogeneous of degree~$1$ feedback control $\Phi:\R^n \to \R^m$ such that the closed-loop~\eqref{eq:modeIndSys} is RUES.\label{item:APropBlanchini}
    \item[(B)] DFS if and only if there exist homogeneous of degree~$1$ feedback controls $\Phi_1,\dots,\Phi_M:\R^n\to \R^m$ such that the closed-loop~\eqref{eq:modeDepSys} is RUES. \label{item:BPropBlanchini}
\end{itemize}
\end{thm}
\textcolor{black}{
The proof of Item~\emph{(A)} 
substantially follows from the results  in~\cite{Blanchini94}. We refer to \cite{BlaMiaSav07} for the argument for the mode-dependent case, i.e., Item~\emph{(B)}. } We want to stress that this result allows us to restrict the search for stabilizing feedback maps into the class of homogeneous of degree $1$ functions\footnote{We recall that it is well-established that simply considering \emph{linear} feedback maps is restrictive in our setting, see the examples provided in~\cite{BlaMiaSav07} and~\cite{HuShen24}.}.
Moreover we also underline that Theorem~\ref{thm:HomogenenousCharachterization} also ensures that the closed-loop can be made \emph{robustly} uniformly exponentially stable via homogeneous feedback controls. Such robustness will play a crucial role in the developments of the main results in Section~\ref{sec:Main}.
\subsection{Memory-based stabilizability notions}\label{subsec:MemoryStabNotion}
In this section we introduce additional notions of stabilizability for~\eqref{eq:SwitchedSystemInput} that involve the notion of \emph{memory}. For a general introduction regarding discrete-time systems depending on a notion of memory inspired by language-theory objects, we refer to~\cite[Section 3]{DelRosJung24}.
Given $M\in\N_+$ let us consider $f:\N\times\M^\star \times \R^n \to \R^n$, and the \emph{system with memory} defined by
\begin{equation}\label{Eq:SystemWithMemory}
x(k+1)=f(k,\sigma_{[0,k]}, x(k)),
\end{equation}
where $\sigma\in \M^\omega$ is a (switching) signal. We denote by $\phi(\cdot\,,\sigma,x_0):\N\to \R^n$ the (unique) solution to~\eqref{Eq:SystemWithMemory} with respect to the signal $\sigma$ and the initial condition $x(0)=x_0$.
\begin{defn}\label{defn:UniformwithMememory}
Given $M\in\N_+$ and $f:\N\times \M^\star \times \R^n \to \R^n$, system~\eqref{Eq:SystemWithMemory} is said to be \emph{uniformly exponentially stable} (UES) if there exist $C\geq 1$ and $\gamma \in [0,1)$ such that
\[
|\phi(k,\sigma,x_0)|\leq C\gamma^k|x_0|\;\;\;\forall x_0\in \R^n,\;\forall \,\sigma\in  \M^\omega, \;\forall k\in \N.
\]
\end{defn}
Note that the term ``\emph{uniformly}'' in Definition~\ref{defn:UniformwithMememory} is referring to both the initial condition $x_0\in \R^n$ and the signal $\sigma\in \M^\omega$: the constants $M$ and $\gamma$ do not depend on either of these parameters. On the other hand, the initial time is always considered to be $k=0$, uniformity in this setting does \emph{not} concern the time variable.

\begin{defn}[Stabilizability with memory]\label{defn:STABwithMemory}
Given $M\in \N_+$ and a set $\cF=\{(A_i,B_i)\in \R^{n\times n}\times \R^{n\times m}\;\vert\;i\in \M\}$, system~\eqref{eq:SwitchedSystemInput} is:
\begin{itemize}
\item[(A):] \emph{Stabilizable via current-mode-independent linear feedbacks with memory} (SIL$_{\text m}$) if there exists $\cK:\N\times \M^\star\to \R^{m\times n}$ such that the system
\[
x(k+1)=A_{\sigma(k)}x(k)+B_{\sigma(k)}\cK(k, [\sigma]_{[0,k-1]})x(k),
\]
is UES, in the sense of Definition~\ref{defn:UniformwithMememory}.

\item[(B):] \emph{Stabilizable via current-mode-dependent linear feedbacks with memory} (SDL$_{\text m}$) if there exists $\cK:\N\times \M^\star\to \R^{m\times n}$ such that the system
\[
x(k+1)=A_{\sigma(k)}x(k)+B_{\sigma(k)}\cK(k, [\sigma]_{[0,k]})x(k),
\]
is UES,  in the sense of Definition~\ref{defn:UniformwithMememory}.

\end{itemize}
In both cases, if there exist $T\in \N$ and a map $\wt \cK:\N\times \M^T \to \R^{m\times n}$ such that  
\begin{equation}\label{eq:FiniteMemory}
\cK(k,[\sigma]_{[0,k]})=\wt \cK(k,[\sigma]_{[k-T,k]}),\;\;\forall\,\sigma\in \M^\omega, \forall\,k\geq T
\end{equation}
the system is said to be \emph{stabilizable (cur.-mod.-ind. or cur.-mod.-dep., respectively) via linear feedbacks with $T$-memory} ($T$-SIL$_{\text m}$ and $T$-SDL$_{\text m}$, respectively). Moreover, if the controllers do not explicitly depend on the time-variable $k\in \N$, we say that they are \emph{time-independent}.
\end{defn}
\begin{comment}
\begin{rem}
The main difference between the notions in (A) and (B) in Definition~\ref{defn:STABwithMemory} is the following: the memory-dependent controller in (A), at each time step $k\in \N$ has access to the past values of the switching signal, but \emph{not} to the current value $\sigma(k)$.   On the other hand, the memory-dependent controller in (B) has access to the past \emph{and current} values of the signal $\sigma$. This can be compared with the two notions (mode-independent and mode-dependent) of robust stability introduces in Definition~\ref{defn:StabNot}.
When these controllers do not depend on the \emph{whole} past-values string, but only on the previous $T$ values, as formalized in~\eqref{eq:FiniteMemory}, we say that the controllers are of ``\emph{$T$-memory type}''. In this setting, the controller does not need to ``remember''  all the past values of $\sigma$ (which will require unbounded memory availability, as $k$ increases), but only the last $T$ steps.
\end{rem}
\end{comment}
\section{Main characterization results}\label{sec:Main}
In this section we provide a characterization of the IFS and DFS notions for~\eqref{eq:SwitchedSystemInput} introduced in Definition~\ref{defn:StabNot} in terms of graph-based Lyapunov conditions.

\subsection{Graph-based Lyapunov functions}

A rich literature exists on graph-based Lyapunov functions for switched systems. Specifically, such structures have been widely studied for the stability of autonomous switched systems (i.e. of the form $x(k+1) = A_{\sigma(k)}x(k)$) under both arbitrary and constrained switching signals $\sigma$. The basic idea is to encode Lyapunov inequalities through labeled graphs, which ensure sufficient conditions for the system to be asymptotically stable. In~\cite{DelRosAlv24}, these results were firstly adapted for \emph{stabilization} purposes.
Since similar formalism is used in what follows, we  review the necessary graph-theoretic tools and we further refer to~\cite{AhmJun:14, PEDJ:16, PhiAthAng,DelRosAlv24} for the interested reader.

Let us then introduce a few basic concepts concerning graphs. A (labeled and directed) \emph{graph} $\cG=(S,E)$ on an alphabet $\M$ is defined by a finite set of \emph{nodes} $S$ and a subset of labeled edges $E\subseteq S\times S\times \M$.\\ 
A graph $\cG=(S,E)$ on the alphabet $\langle M \rangle$ is \emph{complete} if for all $a \in S$, for all $i \in \langle M \rangle$, there exists at least one node $b \in S$ such that the edge $e=(a,b,i) \in E$.
This definition of completeness is borrowed from automata theory, and requires that from any node there is at least one outgoing edge labeled by $i$, for every symbol $i\in \M$. It is thus different from the classical notion of completeness of (unlabeled) graphs, requiring the existence of all edges.\\
Moreover, a graph $\cG=(S,E)$ on the alphabet $\M$ is said to be \emph{deterministic} if, for all $a\in S$ and $i\in \M$ there exists at most one $b\in \M$ such that $e=(a,b,i)\in E$.

 The following two propositions from~\cite{DelRosAlv24} characterize sufficient conditions in terms of complete (and deterministic) graphs for mode-independent (and dependent) stabilization using PWL feedbacks. Those results will be needed later and are recalled for the sake of completeness. 

\begin{prop}[\cite{DelRosAlv24},  Proposition~10]\label{Prop:RobustMachin}
Consider $M\in \N_+$ and a set $\cF=\{(A_i,B_i)\in \R^{n\times n}\times \R^{n\times m}\;\vert\;i\in \M\}$, and a \emph{complete} graph $\cG=(S,E)$ on $\M$. Suppose there exist $\{P_s\}_{s\in S}\subseteq \bS_+^{n\times n}$, $\{K_s\}_{s\in S}\subseteq\R^{m\times n}$ such that
\begin{equation*}
(A_i+B_iK_a)^\top P_b (A_i+B_iK_a)-P_a\prec 0,\;\;\forall e=(a,b,i)\in E,
\end{equation*}
then the piecewise linear map
\begin{equation}\label{eq:robustcontrol}
    \Phi(x):=K_{\kappa(x)}x,~x\in \R^n
\end{equation}
exponentially stabilizes system~\eqref{eq:SwitchedSystemInput} (in the mode-independent sense of Item~\emph{1.} of Definition~\ref{defn:StabNot}), where $\kappa:\R^n\to S$ is any homogeneous of degree $0$ function satisfying\footnote{If  $S$ is ordered, we set $\kappa(x)=\min_S\{\argmin_{s\in S} \{x^\top P_sx\}\}$, where $\min_S(\xbar{S})$ is the minimal element of $\xbar{S}\subseteq S$, w.r.t. the order of~$S$.}
\begin{equation}\label{eq:argminDefinition}
\kappa(x)\in\argmin_{s\in S} \{x^\top P_sx\},\;\;\;\forall x\in \R^n.
\end{equation}
\end{prop}

\begin{prop}[\cite{DelRosAlv24}, Proposition~16]\label{Prop:COmpleteModeDependent}
Consider $M\in \N_+$, a set $\cF=\{(A_i,B_i)\in \R^{n\times n}\times \R^{n\times m}\;\vert\;i\in \M\}$ and a \emph{complete and deterministic} graph $\cG=(S,E)$ on $\M$. Suppose there exist $\{P_s\}_{s\in S}\subseteq \bS_+^{n\times n}$, $\{K_{s,j}\}_{(s,j)\in S\times \M}\subseteq \R^{m\times n}$ such that
\begin{equation*}
(A_i+B_iK_{a,i})^\top P_b (A_i+B_iK_{a,i})-P_a\prec 0,\;\;\forall e=(a,b,i)\in E,
\end{equation*}
then, the piecewise linear maps
\begin{equation}\label{eq:modedepcontrol}
\Phi_i(x):=K_{\kappa(x),i}\,x,~\text{for}~i\in \M
\end{equation}
 exponentially stabilize system~\eqref{eq:SwitchedSystemInput} (in the mode-dependent sense of Item~\emph{2.} of Definition~\ref{defn:StabNot}) where $\kappa:\R^n\to S$ is any homogeneous of degree $0$ function satisfying~\eqref{eq:argminDefinition}.
\end{prop}

The proofs of Propositions~\ref{Prop:RobustMachin} and~\ref{Prop:COmpleteModeDependent} can be found in~\cite{DelRosAlv24}. We note here that robustness  can be proven using Lemma~\ref{lemma:technicalLemma}, since it can be proved that the function $V:\R^n\to \R$ defined by
$
V(x):={\tiny \min_{s\in S}\left \{\sqrt{x^\top P_sx}\right\}}$,
is a continuous exponential Lyapunov function for $F^\Phi$ (respectively, $F^\Phi_d$), in the sense of Definition~\ref{defn:LyapunovExponential}. 

% \begin{rem}
%     In~\cite{DelRosAlv24}, Propositions~\ref{Prop:RobustMachin} and~\ref{Prop:COmpleteModeDependent} were leveraged to derive LMI conditions~\cite[Corollary 12 \& Corollary 16]{DelRosAlv24} for the design of the stabilizing feedback gains by using a hierarchy of De-Bruijn graphs of order $l\in \N$, which are complete and deterministic. In the mode-independent case, the LMIs were conjectured to be both \emph{sufficient} and \emph{necessary} for some graph order $l\in \N$. That conjecture remains unproven. In the following section we show equivalence considering a novel and richer graph structure.
% \end{rem}

\subsection{Feedback-tree graphs}

In order to provide our main result, we define a hierarchical family of graphs that will be used for encoding the numerical conditions. A graph in this class will be called a \emph{feedback-tree graph} (FTG) for convenience. 

To provide the formal definition, let us introduce the following notation: given $M \in \N_+$ and $T \in \N$, we consider the set $ \mathcal{W}(M,T) :=  \bigcup_{k=0}^{T} \M^{k}$,
 i.e.. the set of strings of length less than $T$ over $\M$, with the convention that $\M^0=\{ \bf \epsilon\}$ where ``${\bf \epsilon}$'' denotes the empty string.

\begin{figure}[!t]

\centering
% \captionsetup{justification=justified,width=\linewidth}
\includegraphics[scale=0.55]{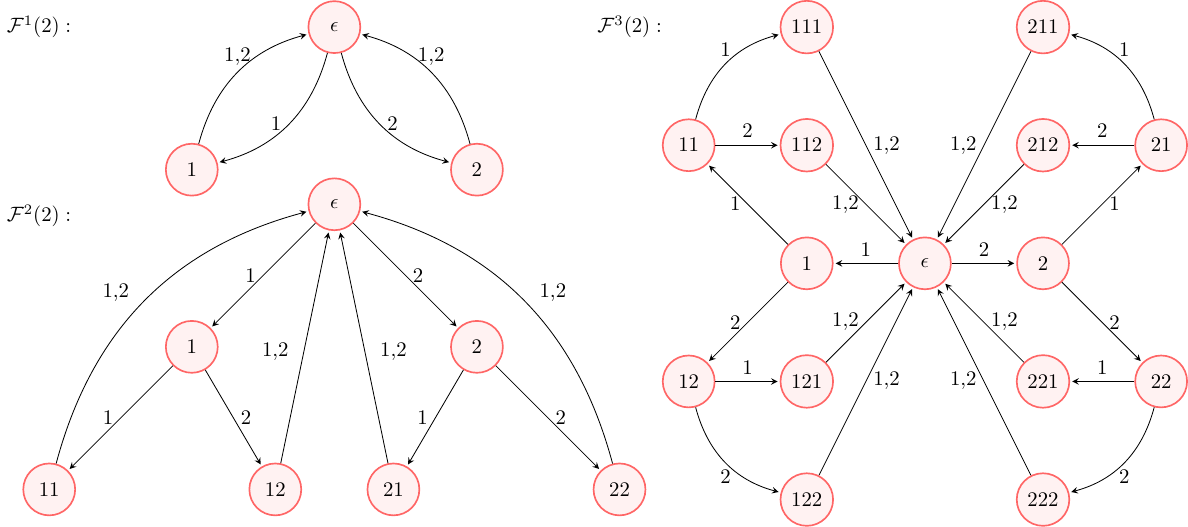}
\caption{The graphs $\mathcal{F}^1(2)$, $\mathcal{F}^2(2)$, and $\mathcal{F}^3(2)$, i.e., the feedback-tree graphs of order $1$, $2$, and $3$ on $\M=\{1,2\}$. The edges labeled by multiple labels can be considered as a graphical shortcut for multiple edges, each one with a single label.}
\label{fig:GraphF}
\end{figure}

\begin{defn}[Feedback-tree graphs]\label{defn:FTG}
 The \emph{feedback-tree graph (FTG) of order $T$ (on $\M$)} is denoted by $\mathcal{F}^{T}(M) = \left(S_{M,T},E_{M,T}\right)$, with $S_{M,T}= \mathcal{W}(M,T)$ and $E_{M,T}$ defined by
 \begin{itemize}
     % \item $(\wi,(\wi,j),j) \in E_{M,T}$ for every $j \in \M$ if $\wi \neq 0$ 
     \item $(\wi,(\wi,i),i) \in E_{M,T}$ for every $i \in \M$ and every $\wi\in \mathcal{W}(M,T-1)= \mathcal{W}(M,T) \setminus \M^T$
     \item  $(\wi,\eS,i) \in E_{M,T}$ for every $i \in \M$ and any $\wi \in \M^T$.
 \end{itemize}
\end{defn}

One can observe that $\mathcal{F}^{T}(M)$ is a complete and deterministic graph for any $M \in \N_+$ and any $T \in \N$, see Fig.~\ref{fig:GraphF}. 
Even if the graph $\cF^T(M)$ is not a tree in a graph-theory sense, the chosen nomenclature is suggested by the following intuition: from the ``root'', i.e., the node labeled by $\eS$, we generate branches in which the visited nodes are labeled by symbols from the alphabet $\M$, up to length $T$. We thus reach~$M^T$ ``leaves'' labeled by words of length $T$. We then connect all these ``leaves'' with the starting ``root'', with edges labeled by any possible symbol in $\M$.

\subsection{Feedback stabilization: equivalent LMIs conditions}\label{sec:PWL}
In this subsection, we  use  the feedback-tree graph structure to provide a characterization result for  robust feedback stabilizability, both in the mode-independent and mode-dependent case.
We first state our sufficient conditions  as a corollary of the Propositions~\ref{Prop:RobustMachin} and~\ref{Prop:COmpleteModeDependent} in the case of the feedback-tree graphs.

\begin{cor}[Sufficiency of FTG-based conditions]\label{cor:FTGModeDep} 
Consider $M\in \N_+$, a set $\cF=\{(A_i,B_i)\in \R^{n\times n}\times \R^{n\times m}\;\vert\;i\in \M\}$, any $T\in \N$, and the  graph $\mathcal{F}^{T}(M) = \left(S_{M,T},E_{M,T}\right)$.

\begin{enumerate}[leftmargin=0.3cm]
\item[$\bullet$]\emph{(Mod-Ind):}
Suppose there exist $\{\xbar{P}_{\wi} \}_{\wi\in S_{M,T}}\subseteq \bS_+^{n\times n}$, $\{\xbar{K}_{\wi}\,\}_{\wi \in S_{M,T}}\subseteq \R^{m\times n}$ such that
\begin{subequations}\label{eq:LMIConditionModeINDFTG}
\begin{equation}\label{eq:LMIConditionModeINDFTG1}
\begin{split}
    \begin{bmatrix}
\xbar{P}_{\wi} & \xbar{P}_{\wi}A_i^{\top}+\xbar{K}_{\wi}^{\top}B_i^{\top} \\ \\ \star & \xbar{P}_{(\wi,i)} \
\end{bmatrix} & \succ 0,\\ \forall\, \wi\in \mathcal{W}(M,T-1) \subseteq S_{M,T},\forall \,i &\in \M,
\end{split}
\end{equation}
\begin{equation}\label{eq:LMIConditionModeINDFTG2}
\begin{split}
    \begin{bmatrix}
\xbar{P}_{\wi} & \xbar{P}_{\wi}A_i^{\top}+\xbar{K}_{\wi}^{\top}B_i^{\top} \\ \\ \star & \xbar{P}_{\eS} \
\end{bmatrix} & \succ 0, \\ \forall\, \wi\in \M^{T} \subseteq S_{M,T},\forall \,i  \in & \M.
\end{split}
\end{equation}
\end{subequations}
Consider the piecewise linear map defined by $\Phi(x):={K}_{\kappa(x)}x$, $x\in \R^n$,  with $K_{\wi}= \xbar{K}_{\wi} \xbar{P}_{\wi}$ and $\kappa:\R^n\to S_{M,T}$ any homogeneous of degree $0$ function satisfying
\begin{equation}\label{eq:Argmin1}
\kappa(x)\in\argmin_{\wi \in S_{M,T}} \{x^\top {P}_{\wi}x\}
\end{equation}
with ${P}_{\wi} = \xbar{P}^{-1}_{\wi}$.
Then, the map $\Phi:\R^n\to \R^m$ exponentially stabilize{\color{black}s} system~\eqref{eq:SwitchedSystemInput}, in the mode-independent sense of Item~\emph{1.} of Definition~\ref{defn:StabNot}.

\item[$\bullet$]\emph{(Mod-Dep)}:
Suppose there exist $\{\xbar{P}_{\wi} \}_{\wi\in S_{M,T}}\subseteq \bS_+^{n\times n}$, $\{\xbar{K}_{\wi,i}\,\}_{\wi \in S_{M,T}, i \in \M}\subseteq \R^{m\times n}$ such that
\begin{subequations}\label{eq:LMIConditionModeDependentFTG}
\begin{equation}\label{eq:LMIConditionModeDependentFTG1}
\begin{split}
    \begin{bmatrix}
\xbar{P}_{\wi} & \xbar{P}_{\wi}A_i^{\top}+\xbar{K}_{\wi,i}^{\top}B_i^{\top} \\ \\ \star & \xbar{P}_{(\wi,i)} \
\end{bmatrix} & \succ 0,\\ 
\forall \wi\in \mathcal{W}(M,T-1) \subseteq S_{M,T},\forall \,i & \in \M,
\end{split}
\end{equation}
\begin{equation}\label{eq:LMIConditionModeDependentFTG2}
\begin{split}
    \begin{bmatrix}
\xbar{P}_{\wi} & \xbar{P}_{\wi}A_i^{\top}+\xbar{K}_{\wi,i}^{\top}B_i^{\top} \\ \\ \star & \xbar{P}_{\eS} \
\end{bmatrix} & \succ 0,\\ 
\forall \wi\in \M^{T} \subseteq S_{M,T},\forall \,i  \in & \M.
\end{split}
\end{equation}
\end{subequations}
For any $i\in \M$, consider the piecewise linear map defined by $\Phi_i(x):={K}_{\kappa(x),i}\,x$, $x\in \R^n$, with $K_{\wi,i}= \xbar{K}_{\wi,i} \xbar{P}_{\wi}$ and $\kappa:\R^n\to S_{M,T}$ any homogeneous of degree $0$ function satisfying~\eqref{eq:Argmin1} with ${P}_{\wi} = \xbar{P}^{-1}_{\wi}$.
Then, the maps $\Phi_1,\dots \Phi_M:\R^n\to \R^m$ exponentially stabilize system~\eqref{eq:SwitchedSystemInput}, in the mode-dependent sense of Item~\emph{2.} of Definition~\ref{defn:StabNot}.
\end{enumerate}
\end{cor}
\begin{proof}
% {\color{black} We explicitly discuss the \emph{mode-dependent case} only. The proof is an application of Propositions~\ref{Prop:RobustMachin} and~\ref{Prop:COmpleteModeDependent} to the case of the feedback-tree graph $\cF^T(M)$, and follows by showing (via standard LMI manipulation techniques) that the existence of matrices $\{\xbar{P}_{\wi} \}_{\wi\in S_{M,T}}\subseteq \bS_+^{n\times n}$ and $\{\xbar{K}_{\wi,i}\,\}_{\wi \in S_{M,T}, i \in \M}\subseteq \R^{m\times n}$ such that~\eqref{eq:LMIConditionModeDependentFTG1}-\eqref{eq:LMIConditionModeDependentFTG2} holds is equivalent to the existence of matrices $\{{P}_{\wi} \}_{\wi\in S_{M,T}}\subseteq \bS_+^{n\times n}$, $\{{K}_{\wi,i}\,\}_{\wi \in S_{M,T}, i \in \M}\subseteq\R^{m\times n}$ such that \begin{equation}
%     \begin{split}
%         (A_i+B_iK_{a,i})^\top P_b (A_i+&B_iK_{a,i})-P_a\prec 0, \\ 
%         \forall e=(a,b,i)&\in E_{M,T},
%     \end{split}
% \end{equation}
% is satisfied, where $E_{M,T}$ is the set of edges of the graph $\mathcal{F}^T(M)$ introduced in Definition~\ref{defn:FTG}. Since all FTGs are by definition~\emph{complete and deterministic} graphs, Proposition~\ref{Prop:COmpleteModeDependent} applies and the proof is concluded.}

We explicitly develop the \emph{mode-dependent case} only. From standard LMI manipulation techniques, the existence of matrices $\{\xbar{P}_{\wi} \}_{\wi\in S_{M,T}}\subseteq \bS_+^{n\times n}$ and $\{\xbar{K}_{\wi,i}\,\}_{\wi \in S_{M,T}, i \in \M}\subseteq \R^{m\times n}$ such that~\eqref{eq:LMIConditionModeDependentFTG1}-\eqref{eq:LMIConditionModeDependentFTG2} are satisfied is equivalent to the existence of matrices $\{{P}_{\wi} \}_{\wi\in S_{M,T}}\subseteq \bS_+^{n\times n}$, $\{{K}_{\wi,i}\,\}_{\wi \in S_{M,T}, i \in \M}\subseteq\R^{m\times n}$ such that the inequalities
\begin{equation}\label{eq:aux1}
\begin{split}
    (A_i+B_iK_{\wi,i})^\top P_{(\wi,i)} (A_i+B_iK_{\wi,i})&-P_{\wi}\prec 0, \\
    \forall \wi\in \mathcal{W}(M,T-1) \subseteq S_{M,T},\forall \,i & \in \M,
\end{split}
\end{equation}
\begin{equation}\label{eq:aux2}
  \begin{split}
        (A_i+B_iK_{\wi,i})^\top P_{\eS} (A_i+B_iK_{\wi,i})& -P_{\wi} \prec 0,\\
        \forall \wi\in \M^{T} \subseteq S_{M,T},\forall \,i & \in \M,
  \end{split}
\end{equation}
hold. The proof then follows by noting that~\eqref{eq:aux1}-\eqref{eq:aux2} can be rewritten as $(A_i+B_iK_{a,i})^\top P_b (A_i+B_iK_{a,i})-P_a\prec 0$ for all $e=(a,b,i)\in E_{M,T}$,
where $E_{M,T}$ is the set of edges of the graph $\mathcal{F}^T(M)$ introduced in Definition~\ref{defn:FTG}. Since all FTGs are by definition~\emph{complete and deterministic} graphs, Proposition~\ref{Prop:COmpleteModeDependent} applies and the proof is concluded. \end{proof}

Corollary~\ref{cor:FTGModeDep} introduces a new hierarchy of LMI conditions for the computation of PWL stabilizing gains for~\eqref{eq:SwitchedSystemInput}. Another hierarchy of LMI conditions has been recently introduced in the literature (see~\cite[Corollary 18]{DelRosAlv24}). \textcolor{black}{The important and novel feature of the FTG-based conditions in Corollary~\ref{cor:FTGModeDep} is that they are \emph{asymptotically correct}, in the sense that they are also~\emph{necessary} for some large enough graph order $T \in \N$. This is formally stated and proved in the sequence. }
\begin{thm}[Characterizations via feedback-trees]\label{thm:necessityFTGmodedependent}
Consider $M\in \N_+$ and a set $\cF=\{(A_i,B_i)\in \R^{n\times n}\times \R^{n\times m}\;\vert\;i\in \M\}$. The following equivalences hold.
\begin{enumerate}[leftmargin=0.3cm]
\item[$\bullet$]\emph{(Mod-Ind):} System~\eqref{eq:SwitchedSystemInput} is IFS (Item \emph{1.} of Definition~\ref{defn:StabNot}) if and only if there exists a $T\in \N$ (large enough) for which the conditions~\eqref{eq:LMIConditionModeINDFTG1}~\eqref{eq:LMIConditionModeINDFTG2} are satisfied.
\item[$\bullet$]\emph{(Mod-Dep):}
System~\eqref{eq:SwitchedSystemInput} is DFS (Item \emph{2.} of Definition~\ref{defn:StabNot}) if and only if there exists a $T\in \N$ (large enough) for which the conditions~\eqref{eq:LMIConditionModeDependentFTG1}~\eqref{eq:LMIConditionModeDependentFTG2} are satisfied.
\end{enumerate}
\end{thm}

\begin{proof}
The ``if parts'' of the proof, i.e., the sufficiency of conditions~\eqref{eq:LMIConditionModeINDFTG1}~\eqref{eq:LMIConditionModeINDFTG2} (conditions~\eqref{eq:LMIConditionModeDependentFTG1}~\eqref{eq:LMIConditionModeDependentFTG2}, respectively) for stabilizability is provided in Corollary~\ref{cor:FTGModeDep}.
It remains to prove the ``only if'' part. To avoid redundancy, we explicitly report here the proof only for the \emph{mode-dependent} case; the mode-independent proof directly follows, \emph{mutatis mutandis}. 
Let us suppose that system~\eqref{eq:SwitchedSystemInput} is DFS, and (via Theorem~\ref{thm:HomogenenousCharachterization}) there exist maps $\Phi_1,\dots \Phi_M:\R^n \to \R^m$ such that the difference inclusion~\eqref{eq:modeDepSys}
 is RUES. By definition of RUES, there exists a continuous and positive definite $\rho:\R^n \to \R_+$ such that the $\rho$-perturbed difference inclusion (recall Definition~\ref{defn:perturbation})
 \begin{equation}\label{eq:InculsionProof}
 \color{black}
x(k+1)\in F^\Phi_{d,\rho}(x)\supseteq F_d^\Phi(x)+\rho(x)\B 
\end{equation}
is UES. This, in turns, implies by definition that there exist
 constants \(\gamma \in [0,1)\) and $C\geq 1$ such that $| \phi(k) | \leq C \gamma^k| x_0 |$
% \begin{equation}\label{eq:expbound}
%     | \phi(k) | \leq C \gamma^k| x_0 |, 
% \end{equation}
for all $k \in \N$, all $x_0\in \R^n$ and all $\phi\in \cS_{F^\Phi_{d,\rho},\,x_0}$. Let $\beta\in (0,1)$ and
consider \(T\in \N\) sufficiently large such that \(C \gamma^{T+1} \le \frac{\beta}{\sqrt{n}}\), leading to
\vspace{-0.3cm}
\begin{equation}\label{eq:DecreasingProof}
| \phi(k) | \le  \frac{\beta}{\sqrt{n}}| x_0 |, 
\end{equation}
$\forall k \geq T+1$, $\forall \,x_0\in \R^n$, $\forall\;\phi\in \cS_{F^\Phi_{d,\rho},\,x_0}$.\\
Let us now define iteratively several sequences of matrices
$\{{X}_{\wi} \}_{\wi\in \mathcal{W}(M,T+1)}\subseteq \R^{n\times n}$,
$\{{K}_{\wi,i}\,\}_{\wi \in \mathcal{W}(M,T), i \in \M}\subseteq \R^{m\times n}$,
$\{{D}_{\wi,i} \}_{\wi \in \mathcal{W}(M,T), i \in \M}\subseteq \R^{n\times n}$ (with iteration on the length $k\in \{0,\dots, T\}$ of the strings $\wi\in \mathcal{W}(M,T+1)$). 
We denote 
${X}_{\wi} =(x_{\wi}^1,\dots,x_{\wi}^n)$ and
${D}_{\wi,i} =(d_{\wi,i}^1,\dots,d_{\wi,i}^n)$
where $x_{\wi}^j,d_{\wi,i}^j \in \R^n$. 
The construction presented below will ensure that
the matrices ${X}_{\wi}$ are non-singular.

Let us define $X_{\epsilon}=I$ \textcolor{black}{and $C_{\wi,i} := (A_i+B_iK_{\wi,i})$, where for $\wi\in \mathcal{W}(M,T)$ and $i\in \M$, $K_{\wi,i}$ is given by} 
\begin{equation}
    \label{eq:matrixK}
    K_{\wi,i}= \left(\Phi_{i}(x_{\wi}^1),\dots,\Phi_{i}(x_{\wi}^n)\right) X_{\wi}^{-1}.
\end{equation}
Then, it is always possible to choose a matrix ${D}_{\wi,i}$ such that the following two conditions holds:
\begin{enumerate}[label=\alph*),leftmargin=*]
    \item $\mspan(C_{\wi,i}X_{\wi} + D_{\wi,i}) =\R^n$ and 
\begin{equation}
\label{eq:matrixD}
    \mspan\left(D_{\wi,i}^\top \right)=\mspan\left({X}_{\wi}^\top C_{\wi,i}^\top\right)^\perp.
\end{equation}
\item For all $j\in \{1,\dots,n\}$
$|d_{\wi,i}^j|\le \rho(x_{\wi}^j)$.
\end{enumerate}
Then, let 
\begin{equation}
    \label{eq:matrixX}
    X_{(\wi,i)}=C_{\wi,i}X_{\wi} + D_{\wi,i}.
\end{equation}
Condition a) ensures that $X_{(\wi,i)}$ is a non-singular matrix. Moreover, from condition b) and from \eqref{eq:matrixK} we get for all $j\in  \{1,\dots,n\}$
\begin{equation} \label{eq:dynmatrix}
\begin{aligned}
x_{(\wi,i)}^j &= C_{\wi,i}x^j_{\wi} + d^j_{\wi,i} \\
&= A_ix^j_{\wi}+B_i\Phi_{i}(x^j_{\wi})+d^j_{\wi,i}
\in F_{d,\rho}^\Phi(x^j_{\wi}).
\end{aligned} 
\end{equation}
We can now define the matrices $Q_{\wi} = X_{\wi}X_{\wi}^\top$, for all $\wi\in \mathcal{W}(M,T+1)$.
Since $X_{\wi}$ is non-singular, it follows that
$Q_{\wi} \in \bS_+^{n\times n}$.
Then, for all $\wi\in \mathcal{W}(M,T)$, and $i\in \M$
\begin{align}
\nonumber
&Q_{(\wi,i)} = X_{(\wi,i)}X_{(\wi,i)}^\top= \left(C_{\wi,i}X_{\wi} + D_{\wi,i}
\right)
\left(C_{\wi,i}X_{\wi} + D_{\wi,i}
\right)^\top\\
\
& 
\label{eq:decrease}
= C_{\wi,i}X_{\wi}X_{\wi}^\top
C_{\wi,i}^\top+D_{\wi,i}D_{\wi,i}^\top\succeq 
C_{\wi,i}Q_{\wi}
C_{\wi,i}^\top
\end{align}
where the second equality follows from \eqref{eq:matrixX} and the third equality follows from \eqref{eq:matrixD}.
Since the matrices $Q_{\wi}\in \bS_+^{n\times n}$, 
\eqref{eq:decrease} is equivalent, by Schur complement, to
\begin{equation}
\label{eq:decrease2}
C_{\wi,i}^\top Q_{(\wi,i)}^{-1}C_{\wi,i} \preceq Q_{\wi}^{-1}.
\end{equation}
Consider now $\wi \in \M^{T+1}$, 
by~\eqref{eq:DecreasingProof} 
and~\eqref{eq:dynmatrix},
we have that
\begin{equation}\label{eq:sqrtn1}
    |x^j_{\wi}|< 
    \frac{\beta}{\sqrt{n}}\;\;\forall j\in \{1,\dots, n\}.
\end{equation}
Then, for any $z\in \R^n$, it holds that $|X_{\wi} z| \le \sum_{j=1}^n |z_j|\; |x_{\wi}^j| \le \frac{\beta}{\sqrt{n}} \sum_{j=1}^n |z_j| \le \beta |z|$
% \begin{align*}
%     |X_{\wi} z| \le \sum_{j=1}^n |z_j|\; |x_{\wi}^j| \le \frac{\beta}{\sqrt{n}} \sum_{j=1}^n |z_j| \le \beta |z|
% \end{align*}
where the first, second and third inequalities follow from the triangular inequality, 
from \eqref{eq:sqrtn1}, and from Cauchy-Schwarz inequality, respectively.
It then follows that
$
X_{\wi}^\top X_{\wi} \preceq \beta^2 I$,
which is equivalent by Schur complement to 
\begin{equation}
\label{eq:decrease3}
Q_{\wi} = X_{\wi} X_{\wi}^\top \preceq \beta^2 I.
\end{equation}
Let $\beta_0 \in (0,1)$ be given by
$\beta_0 = \beta^{\frac{2}{T+1}}$ and let us define for all $\wi\in \mathcal{W}(M,T)$
${P}_{\wi} = \beta_0^{\ell(\wi)} Q_{\wi}^{-1}$ 
where  $\ell(\wi)$ denote the length of $\wi$.
Let us show that inequalities \eqref{eq:aux1}-\eqref{eq:aux2} hold for this specific choice of matrices.
Let us consider $\wi\in \mathcal{W}(M,T-1)$ and $i\in \M$. Then, from \eqref{eq:decrease2}, we get
\begin{align*}
&C_{\wi,i}^\top P_{(\wi,i)} C_{\wi,i}=\beta_0^{\ell(\wi)+1}C_{\wi,i}^\top Q_{(\wi,i)}^{-1} C_{\wi,i}\\
&\preceq \beta_0^{\ell(\wi)+1} Q_{\wi}^{-1}
= \beta_0 P_{\wi}.
\end{align*}
Hence, \eqref{eq:aux1} holds. 
Consider $\wi\in \M^{T}$ and $i\in \M$;
then
\begin{align*}
 C_{\wi,i}^\top P_{\eS} C_{\wi,i} & = C_{\wi,i}^\top  C_{\wi,i}\\
& \preceq \beta_0^{T+1} C_{\wi,i}^\top  
Q_{(\wi,i)}^{-1}
C_{\wi,i} \preceq \beta_0^{T+1} Q_{\wi}^{-1}
= \beta_0 P_{\wi}
\end{align*}
where the first and second inequalities follow from
\eqref{eq:decrease2} and \eqref{eq:decrease3}, respectively.
Hence, \eqref{eq:aux2} holds. 
Summarizing, we have proved the feasibility (for large enough $T$) of inequalities \eqref{eq:aux1}-\eqref{eq:aux2}, which is also equivalent to the feasibility of the LMIs~\eqref{eq:LMIConditionModeDependentFTG1}-\eqref{eq:LMIConditionModeDependentFTG2}. This concludes the proof.
\end{proof}
\textcolor{black}{
We want to underline the robustness result of Theorem~\ref{thm:HomogenenousCharachterization} is instrumental in the proof of Theorem~\ref{thm:necessityFTGmodedependent}. More precisely, it  enables a constructive proof of the existence of a solution to the LMIs conditions; this is explicitly reflected in the construction given in~\eqref{eq:DecreasingProof},~\eqref{eq:dynmatrix} and~\eqref{eq:sqrtn1}.}

In the following subsection, we will analyze the relationship between Theorem~\ref{thm:necessityFTGmodedependent} and the design of controllers depending on the past values of the switching signal (but not on the past values of the state), to achieve memory-based stabilization as introduced in Subsection~\ref{subsec:MemoryStabNotion}

\subsection{Memory-based linear controllers}\label{sub:MemoryCOntroller}
 In Theorem~\ref{thm:necessityFTGmodedependent} we have shown that the LMI conditions~\eqref{eq:LMIConditionModeINDFTG} (resp.~\eqref{eq:LMIConditionModeDependentFTG}) based on feedback-tree structures, are equivalent, for $T$ large enough, to  feedback stabilization (mod.-ind. and mod.-dep., respectively). The obtained feedback maps were  in general nonlinear (but piecewise linear). In this subsection we show that they also induce the construction of memory-based \emph{linear} feedback.
We first need to introduce some notation.
Given $T\in\N$, we consider the \emph{Euclidean remainder function} $\re_T:\N\to \N$, which, to any natural number $k$ associates its remainder with respect to the division by $T$. More explicitly: $\re_T(k)=0$ if $k=mT$ for some $m\in \N$, $\re_T(k)=1$ if $k=mT+1$ for some $m\in \N$, and so on, so forward. 
Let us now define yet another auxiliary function: for any $T\in\N$ we consider the \emph{cutting function (of length $T$)} $\cut_T:\M^\omega\times \N\to \bigcup_{k=0}^{T-1} \M^{k}$ defined by
\begin{equation}\label{eq:CutDefinition}
\begin{aligned}
\cut_T(\sigma,k)&=[\sigma]_{[k-\re_T(k),\,k-1]}
\end{aligned}
\end{equation}
with the convention that $[\sigma]_{[k,k-1]}=\eS$ for any $\sigma\in \M^\omega$, and any $k\in \N$.
Note that, for any $T\in \N$, $\cut_T$ can be seen as a \emph{causal} function, in the sense that  $\cut_T(\sigma,k)$ depends only on $[\sigma]_{[0,k-1]}$ for all $\sigma\in \M^\omega$ and it is thus independent from the values that $\sigma$ takes after the time $k\in \N$.
\begin{comment}
We clarify the definition with an example:
consider $\sigma\in \M^\omega$;  the cutting function of length~$3$ reads $\cut_3(\sigma,k)=$
{\small\[
\begin{aligned}
\begin{cases}
\eS&\text{if } k \text{ divisible by $3$},\\
\sigma(k-1)&\text{if } k=3m+1, \text{ for some $m \in\N$ },\\
(\sigma(k-2),\sigma(k-1))&\text{if } k=3m+2, \text{ for some $m \in\N$ }.
\end{cases}
\end{aligned}
\]}
In other words, the sequence $k\mapsto\cut_3(\sigma,k)$ is: \,
$\epsilon,\,\sigma(0),(\sigma(0),\sigma(1)),\,\epsilon,\,\sigma(3),(\sigma(3),\sigma(4)),\,\epsilon,\,\dots$.
\end{comment}
We can now state the main result of this subsection.

\begin{prop}\label{prop:MemoryImplication} 
Consider $M\in \N_+$, a set $\cF=\{(A_i,B_i)\in \R^{n\times n}\times \R^{n\times m}\;\vert\;i\in \M\}$, any $T\in \N$, and the  graph $\mathcal{F}^{T}(M) = \left(S_{M,T},E_{M,T}\right)$.
\begin{itemize}[leftmargin=0.3cm]
\item[$\bullet$]\emph{(Mod-Ind):} Let us suppose that the conditions in~\eqref{eq:LMIConditionModeINDFTG} are satisfied, and consider the matrices $K_{\wi}= \xbar{K}_{\wi} \xbar{P}_{\wi}\in \R^{m\times n}$, for any $\wi\in\mathcal{W}(M,T)$.  Then system~\eqref{eq:SwitchedSystemInput} is $T$-SIL$_{\text m}$ (recall Definition~\ref{defn:STABwithMemory}) by considering $\cK:\N\times \M^\star\to \R^{m\times n}$ defined as follows:
\begin{equation}\label{eq:ModeINDepMemoryCOntroller}
\cK(k,\sigma_{[0,k-1]})=K_{\cut_{T+1}( \sigma,k)}, \forall \sigma\in \M^\omega, \forall k\in \N. 
\end{equation}
\item[$\bullet$]\emph{(Mod-Dep):} Let us suppose that the conditions in~\eqref{eq:LMIConditionModeDependentFTG} are satisfied, and consider the matrices $K_{\wi,i}= \xbar{K}_{\wi,i} \xbar{P}_{\wi}\in \R^{m\times n}$, for any $\wi\in\mathcal{W}(M,T)$ and any $i\in \M$.  Then system~\eqref{eq:SwitchedSystemInput} is $T$-SDL$_{\text m}$ (recall Definition~\ref{defn:STABwithMemory}) by considering $\cK_d:\N\times \M^\star\to \R^{m\times n}$ defined as follows:
\begin{equation}\label{eq:ModeDepMemoryCOntroller}
\hspace*{-0.55cm}
\cK_d(k,\sigma_{[0,k]})=K_{(\cut_{T+1}(\sigma, k),\sigma(k))}, \forall \sigma\in \M^\omega, \forall k\in \N. 
\end{equation}
\end{itemize}
\end{prop}

\begin{proof}
The proof is straightforward since mostly notational. We sketch the mode-dependent case only; as always, the mode-independent case can be easily obtained, \emph{mutatis mutandis}. \textcolor{black}{From the proof of} Corollary~\ref{cor:FTGModeDep}, conditions~\eqref{eq:LMIConditionModeDependentFTG} imply that~\eqref{eq:aux1}-\eqref{eq:aux2} are satisfied with ${P}_{\wi} = \xbar{P}^{-1}_{\wi}$ for all $\wi\in \cW(M,T)$.
Suppose $T>0$ (otherwise the proof is trivial) and consider any $x_0\in \R^n$ and any $\sigma\in \M^\omega$.
Let us suppose that $\sigma(0)=i\in \M$. Then, by~\eqref{eq:aux1} we have
\[
\begin{aligned}
  &(A_i+B_i\cK_d(0,\sigma_{[0,0]}))^\top P_{i} (A_i+B_i\cK_d(0,\sigma_{[0,0]}))\\&=(A_i\hspace{-0.08cm}+\hspace{-0.08cm}B_iK_{(\cut_{T+1}( \sigma,0),\,\sigma(0))})^\top P_{i} (A_i\hspace{-0.08cm}+\hspace{-0.08cm}B_iK_{(\cut_{T+1}( \sigma,0),\,\sigma(0))})\\
&=(A_i+B_iK_i)^\top P_{i} (A_i+B_iK_i)\prec P_{\eS}.
  \end{aligned}
\]
Iterating, suppose that $\sigma(1)=j\in \M$, by~\eqref{eq:aux1} we have
 \[
\begin{aligned}
  &(A_j+B_j\cK_d(1,\sigma_{[0,1]}))^\top P_{(i,j)} (A_j+B_j\cK_d(1,\sigma_{[0,1]}))=\\
  &(A_j\hspace{-0.08cm}+\hspace{-0.08cm}B_jK_{(\cut_{T+1}( \sigma,1),\,\sigma(1))})^\top P_{(i,j)} (A_j\hspace{-0.08cm}+\hspace{-0.08cm}B_jK_{(\cut_{T+1}( \sigma,1),\,\sigma(1))})\\
&=(A_j+B_jK_{([\sigma]_{[0,0],j)}})^\top P_{(i,j)} (A_j+B_jK_{([\sigma]_{[0,0],j)}})\\& =(A_j+B_jK_{(i,j)})^\top P_{(i,j)} (A_i+B_iK_{(i,j)})\prec P_i.
  \end{aligned}
\]
We can then iterate up to step $T$ (where in the last iteration we make use of inequality~\eqref{eq:aux2}), proving that there exists a $\gamma\in [0,1)$ such that
        $\phi(T,\sigma,x_0)^\top P_{\eS}\phi(T,\sigma,x_0)\leq \gamma x_0^\top P_{\eS}x_0,$ for all $\sigma\in \M^\omega, x_0\in \R^n$. This, by iteration, also implies 
      $\phi(mT,\sigma,x_0)^\top P_{\eS}\phi(mT,\sigma,x_0)\leq \gamma^m x_0^\top P_{\eS}x_0,$ for all  $ \sigma\in \M^\omega$, $ x_0\in \R^n$, and all $m\in \N$.
In other words, the function $V:\R^n\to \R$ defined by $V(x):=\sqrt{x^\top P_{\eS}x}$ is a \emph{global finite-step Lyapunov function}, see the definition in~\cite[Definition 6]{GeiGie14}, absolutely homogeneous of degree 1. Thus, by adapting the proof of~\cite[Theorem 7]{GeiGie14} implies that system~\eqref{eq:SwitchedSystemInput} equipped with the memory-dependent linear controller in~\eqref{eq:ModeDepMemoryCOntroller} is uniformly exponentially stable (in the sense of Definition~\ref{defn:STABwithMemory}). Moreover, since $\cK(k,\sigma_{[0,k]})=K_{(\cut_{T+1}( \sigma,k-1),\,\sigma(k))}$ and the function $\cut_{T+1}$ defined in~\eqref{eq:CutDefinition} only depends, at most, on $[\sigma]_{[k+1-T-1,k]}=[\sigma]_{[k-T,k]}$ we have that the system is stabilizable mode dependently via linear feedbacks with $T$-memory ($T$-SDL$_{\text m}$), concluding the proof.
\end{proof}
\begin{rem}[Automata-based implementation]\label{rem:AutomataImpl}
    Another way to interpret and implement the results of Corollary~\ref{cor:FTGModeDep} is by automata-based theory. In fact, consider an automata constructed with the states and transitions given by the FTG, with initial state $\{\eS\}$. Satisfaction of the conditions in Corollary~\ref{cor:FTGModeDep} implies that the control $\Phi : \R^n \times S_{M,T} \times \M\to \R^m$ given by $ \Phi_{\wi,\sigma}(x) \coloneqq K_{\wi,\sigma}x$, $x \in \R^n$, $\wi \in S_{M,T}$, $\sigma \in \M$ is also a stabilizing controller for~\eqref{eq:SwitchedSystemInput}, with the Lyapunov function $ V_{\wi(k)}(x(k)) ={\tiny \sqrt{x^{\top}(k)P_{\wi(k)} x(k)}}$, where $\wi(k)$ is the state of the automata based defined by the FTG. In this case, the conditions of Corollary~\ref{cor:FTGModeDep} lead to the exponential stability of the \emph{set} $\{0\} \times S_{M,T}$ for the hybrid system with augmented state $\begin{bmatrix}
        x(k)^{\top} & \wi(k) 
    \end{bmatrix}^{\top}$. 
\end{rem}

\subsection{Comparison with existing results}\label{subsec:DeBrunjii}
\textcolor{black}{In this work,} we have provided sufficient and necessary LMIs conditions for feedback exponential stabilizability of switched systems, both in a mode-dependent and mode-independent settings. Moreover, we related the proposed conditions with a notion of memory-based stabilizability.

\textcolor{black}{It is well known that stabilizability of uncertain or switched linear systems under arbitrary switching can be characterized by the existence of polyhedral contractive sets, yielding set-induced Lyapunov functions; see, e.g.,~\cite{Blanchini94}. Moreover, since any compact convex polytope can be approximated arbitrarily well by unions (or convex hulls) of ellipsoids, one can also prove that a stabilizable system always admits a Lyapunov function of the \emph{min-of-quadratics} type  (see~\cite{HuBla10}). However, this observation alone does \emph{not} imply the necessity result established in this paper. The key technical point is that not every min-type (or piecewise quadratic) Lyapunov function can be represented as a \emph{graph-based} (a.k.a.\ path-complete) Lyapunov function with a finite number of nodes. As observed in~\cite{AhmJun:14,PhiAthAng}, the class of Lyapunov functions encoded by path-complete graphs is strictly smaller and more structured than the class of general min/max of quadratic functions. In particular, even if a stabilizable closed-loop system admits a min-of-quadratics Lyapunov function obtained from the approximation of a polyhedral invariant set, there is no \emph{a priori} guarantee that such a Lyapunov function satisfies a finite set of quadratic inequalities that can be encoded by a graph and verified through convex LMIs. Furthermore, unstructured min/max quadratic Lyapunov functions cannot, in general, be exactly characterized via the $S$-procedure, which only provides sufficient (and not necessary) conditions. Therefore, the necessity results in this work cannot be obtained as a direct consequence of existing polyhedral or min-of-quadratics constructions. Instead, it relies on the explicit construction of a hierarchy of graphs whose associated graph-based Lyapunov inequalities become asymptotically tight, together with a proof strategy that  exploits the underlying graph structure.}

\textcolor{black}{Concerning previous graph-based strategies, in the paper by~\cite{DelRosAlv24}, partially generalizing previous results~(\cite{BLIMAN2003325,LeeKha09,EssLee14}), other stabilization schemes were proposed, also based on graph-based Lyapunov functions. Nonetheless, these results were only sufficient and lacked necessity, as discussed in the introduction. The corresponding LMIs were based on a graph structure known as De Bruijn graphs (see~\cite{DeBru46} for the historical reference), which are complete and deterministic graphs, and are recalled below for use in the following discussion.} 

\begin{defn}[De Bruijn graphs]\label{defn:DeBRunjii}
Given $M \in \N_+, T\in \N$ the \emph{De Bruijn graph of order $T$ (on $\M$)} denoted $\cG_{DB}^T=(S_{d,T},E_{d,T})$ is defined by: $S_{d,T}:=\M^{T}$ and, given any node $\wi=(i_0, \dots,i_{T-1})\in S_{d,T}$, we have $(\wi,\wj, h)\in E_{d,T}$ for every $\wj=(j,i_0,\dots, i_{T-2})$, for any $j\in \M$.
% Given $M,l\in \N$  the \emph{dual De Bruijn graph of order $l-1$ (on the alphabet $\M$)} denoted by $\cH_d^l(M)=(S,E)$ is the dual of $\cH^l(M)$. It is explicitly defined as follows: $S:=\M^{l-1}$ and, given any node $\wi=(i_1, \dots,i_{l-1})\in S$, we have $(\wi,\wj,\wi^f)\in E$ for every $\wj$ of the form $\wj=(h,i_1,\dots, i_{l-2})$, for any $h\in \M$.
\end{defn}

Such class of graphs naturally arises in language theory, since they exploit, in their formal definition, the idea of developing words of an alphabet starting from  ``small'' strings of a given length. In our context, these structures have led to the stabilization results presented in~\cite{DelRosAlv24}. Moreover, as proven and developed in~\cite{LeeDull06,LeeKha09,EssLee14}, such conditions are also related and in fact equivalent (see Theorem~9 in~\cite{LeeKha09}) to a notion of \emph{autonomous/time-independent} finite memory stabilization via linear feedbacks. Indeed, when stabilizability conditions based on De Bruijn graphs are satisfied, there exists a finite number of linear feedback for which it suffices to observe the last $T$ (for a $T$ large enough) values of the switching signal to correctly apply the right linear feedback gain. \textcolor{black}{This might be compared with our Definition~\ref{defn:STABwithMemory}, in which the memory-dependent controllers might depend also on the current instant of time $k$ and/or on the whole ``tail'' of past switched signal values. The linear controllers we provided in Subsection~\ref{sub:MemoryCOntroller} (based on the necessary and sufficient conditions presented in Corollary~\ref{cor:FTGModeDep}) indeed are intrinsically dependent on the current instant time instant $k$, and thus require, in general, more information with respect to controllers based on the De Bruijn structures, which, as said, depend only on a fixed-length tail of past switching values. This might confirm the evidence that conditions formulated using De Bruijn graphs (and hence involving a notion of autonomous memory) may be only sufficient for stabilizability, whereas the conditions based on feedback-tree graphs introduced in this paper have been shown to be both necessary and sufficient. On the other hand, from a numerical perspective, we do not claim any computational superiority for the conditions proposed in this manuscript, whose structure was primarily inspired and guided by the necessity proof discussed above. In the subsequent section, we highlight this computational comparison, by presenting some numerical examples.}

% It is easy to see that such graphs are complete and deterministic, see Figure~\ref{fig:GraphDeB} for a graphical representation.

% \begin{figure}[!t]
% \centering
% \captionsetup{justification=justified,width=\linewidth}
% \includestandalone[mode=buildnew, scale=0.6]{Figures/graphDB}
% \caption{The graphs $\cG_{DB}^2$ on the alphabet $\{1,2\}$.}
% \label{fig:GraphDeB}
% \end{figure}

\section{Numerical examples}\label{sec:num}

In this section, we provide some numerical examples in order to explore the controller design conditions in Corollary~\ref{cor:FTGModeDep}. In order to draw comparisons with respect to the recent literature, let us introduce the following notion.
A constant $\gamma \in [0,\infty)$ is an \emph{attainable growth rate} of~\eqref{eq:SwitchedSystemInput} if
\begin{enumerate}[leftmargin=*]
    \item \emph{(mode-ind.)}  there exists a feedback control $\Phi:\R^n \to \R^m$ and $C \in [1,\infty)$ such that $|\phi(k)|\leq C\gamma^k |x_0|$, for all $x_0\in \R^n$, all $\phi\in \cS_{F^{\Phi},x_0}$, all $k\in \N$.
    \item \emph{(mode-dep.)} there exist feedback controls $\Phi_1,\dots,\Phi_M:\R^n\to \R^m$  and $C \in [1,\infty)$ such that
$|\phi(k)|\leq C\gamma^k |x_0|$, for all $x_0\in \R^n$, all $\phi\in \cS_{F_d^{\Phi},x_0}$, all $k\in \N$.
\end{enumerate}
Moreover, the \emph{infimum} over all attainable mode-independent (respectively, mode-dependent) growth rates for system~\eqref{eq:SwitchedSystemInput} is denoted by $\gamma^I_*$ (respectively, $\gamma^D_*$).

Let us note that if $\gamma^I_*\in [0,1)$ (respectively, $\gamma^D_*\in [0,1)$) then the definition of stability in Definition~\ref{defn:StabNot} is satisfied. In this case, if $\gamma^I_*\in [0,1)$ (respectively, $\gamma^D_*\in [0,1)$) they are also known as \emph{best decay rate}.
We note that the design conditions presented in Corollary~\ref{cor:FTGModeDep} guaranteeing robust exponential stabilization do not explicitly include a parameter modeling the decay rate $\gamma$. On the other hand, if one desires to calculate a control strategy and a Lyapunov function guaranteeing an attainable rate of $\gamma$ for system~\eqref{eq:SwitchedSystemInput}, it suffices to solve the conditions in Corollary~\ref{cor:FTGModeDep} using the modified set of matrices $\cF_{\gamma}=\{(\frac{A_i}{\gamma},\frac{B_i}{\gamma})\in \R^{n\times n}\times \R^{n\times m}\;\vert\;i\in \M\}$ instead of $\cF$ (see, for example,~\cite[Lemma 2.1]{HuShen24}). We will thus use, in the following examples, the estimation of  $\gamma^I_*$, $\gamma^D_*$ provided by our conditions as a ``measure'' of the effectiveness of the proposed  techniques.

\subsection{Example 1 (Mode-ind.)}\label{example:2}

Consider system~\eqref{eq:SwitchedSystemInput} with matrices
\begin{gather*}
    A_1 = \begin{bmatrix}  \phantom{-}0 & 1 \\ -1 & 0  \end{bmatrix},  A_2 = \begin{bmatrix}  -1 & 0 \\ \phantom{-}0 & -0.95  \end{bmatrix},  B_1 = B_2 = \begin{bmatrix}  1 \\ 0  \end{bmatrix}.
\end{gather*}
In this example, we aim to illustrate the results of the two possible controller implementations, i.e., the PWL implementation of Section~\ref{sec:PWL}, and the memory/automata based implementation of Section~\ref{sub:MemoryCOntroller}. Solving the conditions~\eqref{eq:LMIConditionModeINDFTG} in Corollary~\ref{cor:FTGModeDep} for the modified set of matrices $\cF_{\gamma}$ with $\gamma=0.9606$ yields a feasible solution for the FTG of order $T=5$. The graph order can be further increased to achieve a smaller guaranteed decay rate (and vice-versa), as it will be illustrated in the next example. The solution for $T=5$ involves 63 matrices $P_{\wi}$ and $K_{\wi}$, which are not explicitly reported in the paper for brevity. The values of the Lyapunov functions $W(x):=\min_{\wi \in S_{M,T}}\{\sqrt{x^\top P_{\wi} x}\}$ (for the PWL implementation) and $V_{\wi(k)}(x(k)) = \sqrt{x^{\top}(k)P_{\wi(k)} x(k)}$ (for the memory/automata based implementation, recall Remark~\ref{rem:AutomataImpl}) evaluated at the trajectory starting in $x_0 = \begin{bmatrix}
    0 & -1 
\end{bmatrix}^{\top}$ for a particular switching sequence $\sigma(k)$ are shown in Fig.~\ref{ex2}. It is interesting to remark that the two implementations, although leading to a stable closed-loop with the same guaranteed convergence rate, result in different trajectories. Intuitively, this happens because the PWL-based implementation depends on the function $\kappa:\R^n\to S_{M,T}$ (introduced in Corollary~\ref{cor:FTGModeDep}), which selects, for each time $x(k)$, the control that minimizes the current ``energy" of the system measured by the Lyapunov function $W$. Such minimization allows for ``jumps" among the nodes of the FTG, therefore leading to different values for $W$ and $V$. From Fig.~\ref{ex2}, it is interesting to note that although such minimization happens at each time $k$, it does not guarantee any order relation between both implementations; i.e., it is not possible to ensure $W(x(k)) \leq V_{\wi(k)}(x(k)), \forall k \in \N$ (or vice versa) along the whole trajectory of the system. This happens because the minimization performed at time $k$ does not take into account the future of the trajectory, considering only the immediate actions.

\begin{figure}[ttb!]
\center{\includegraphics[width=0.7\linewidth,angle=0]{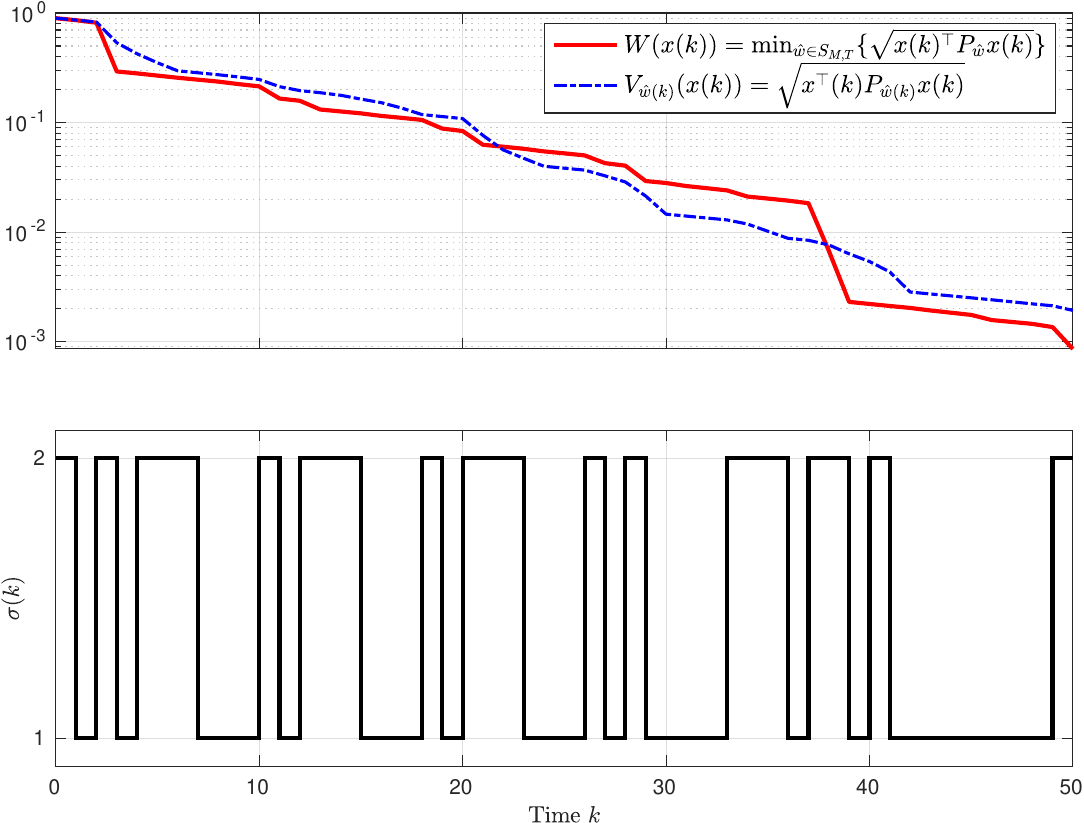}
\caption{$W(x(k))$, $V_{\wi(k)}(x(k))$, and the switching signal $\sigma(k)$ for \emph{Example~\ref{example:2}}, plotted in \emph{log} scale for improved visualization.}
\label{ex2}}
\end{figure}

\subsection{Example 2 (Mode-dep.)}\label{example:3} 
Consider system~\eqref{eq:SwitchedSystemInput} with matrices 
\begin{equation*}
   A_1 = \begin{bmatrix}
        1  & 0 \\
    1  & 1
    \end{bmatrix},\; A_2 = \begin{bmatrix}
         1 &   1 \\
   0   & 1
    \end{bmatrix},\; B_1 = \begin{bmatrix}
        1 \\ 0
    \end{bmatrix},\; B_2 = \begin{bmatrix}
        0 \\ 0
    \end{bmatrix}.
\end{equation*}
This bimodal system was recently studied in~\cite[Example 4.3]{HuShen24}, where it was demonstrated that $\gamma^D_* \approx
1.2493$. One can find upper bounds on $\gamma^D_*$ by solving the conditions~\eqref{eq:LMIConditionModeDependentFTG} of Corollary~\ref{cor:FTGModeDep} using $\cF_{\gamma}$ while performing a dichotomy/bisection search over $\gamma$. In Table~\ref{tab:exampleDecayRate}, we summarize the relationship between the FTG of order $T$ and such upper bounds. For comparison, the best upper bound solving an LMI design problem for mode-dependent controllers in~\cite[eq.~(13)]{HuShen24} was $\gamma^D_* \leq 1.4143$. \cite{HuShen24} also cites that the use of a design algorithm in~\cite[Theorem 4]{DAAFOUZ2001355} leads to $\gamma^D_* \leq 1.3305$. Remarkably, even the FTG of order $T=1$ leads to an improved bound in comparison with the aforementioned methods. On the other hand, solving the same numerical problem using the mode-dependent design LMIs based on De Bruijn graphs presented in~\cite[Corollary 18]{DelRosAlv24} leads to improved results in comparison to the performance obtained with the FTG, as summarized in Table~\ref{tab:exampleDecayRateDeBruijn}. To better illustrate the powerful approximate results of the conditions from~\cite[Corollary 18]{DelRosAlv24} (discussed in subsection~\ref{subsec:DeBrunjii}), we included one more significant digit.
Note that $\gamma^D_* \approx
1.2493$ is achieved by $T=6$. 
It thus turns out that for this specific numerical example the graph-based conditions of~\cite[Corollary 18]{DelRosAlv24} are providing a tighter estimation of the minimal attainable growth rate. We want to underline again that, despite this ``local'' phenomenon, these conditions are only known to be sufficient, while the conditions of~Proposition~\ref{cor:FTGModeDep} proposed in this paper have been proven to be (asymptotically) exact, recall Theorem~\ref{thm:necessityFTGmodedependent}. It is anyhow interesting to note that, despite this ``theoretical advantage'', in specific examples other sufficient conditions can perform better, at least when the length of the considered graphs is bounded.

\begin{table}[h!]
\centering
\caption{Relation between upper bound on $\gamma^D_*$ and FTG graph order $T$ for Example~\ref{example:3}.}
\begin{tabular}{l|lllll}
                           &\cellcolor{gray!20}   $T=1$ & \cellcolor{gray!20} $T=3$ & \cellcolor{gray!20} $T=5$ &\cellcolor{gray!20}  $T=8$ & \cellcolor{gray!20} $T=11$  \\ 
\hline
$\gamma^D_* \leq$ & ~1.3289  & ~1.3047  & ~1.2943  & ~1.2735   & ~1.2713  
\end{tabular}\label{tab:exampleDecayRate}
\end{table}
\vspace{-0.3cm}
\begin{table}[ht!]
\centering
\caption{Relation between upper bound on $\gamma^D_*$ and De Bruijn graph order $T$ from~\cite[Corollary 18]{DelRosAlv24} for Example~\ref{example:3}.}
\begin{tabular}{l|llll}
                           &\cellcolor{gray!20}   $T=1$ & \cellcolor{gray!20} $T=3$ & \cellcolor{gray!20} $T=5$ &\cellcolor{gray!20}  $T=6$\\ 
\hline
$\gamma^D_* \leq$ & 1.32472  & 1.25843  & 1.24952  & 1.24934 
\end{tabular}\label{tab:exampleDecayRateDeBruijn}
\end{table}
\begin{figure*}[h!]
\centering
\begin{tikzcd}
 {} & {\exists T\in \N \text{ s.t.~\eqref{eq:LMIConditionModeINDFTG} feas. }} & {} & {} & {} & {\exists\, T\in \N \text{ s.t. $\cG_{DB}^T$-LMIs}} \\
 {\text{Robust stab.}} & {\text{IFS (Def.~\ref{defn:StabNot})}} & {} & {\parbox{1.7cm}{\centering $\exists\, T\in \N$ s.t.\\ $T$-$\mathrm{SIL}_m$}}
 & {} & {\parbox{2.6cm}{\centering $\exists\, T\in \N$ s.t.\\ $T$-$\mathrm{SIL}_m$ time ind.}}
  \arrow["\text{Prop.~\ref{prop:MemoryImplication}}" {pos=0.6}, shift left=0, Leftrightarrow, from=1-2, to=2-4]
  \arrow[""{pos=0.6}, shift left=0, Leftrightarrow, from=2-4, to=2-2]
  \arrow[shorten <=2pt, shorten >=4pt, Rightarrow, from=2-6, to=2-4]
  \arrow["{\text{Thm. 9,~\cite{LeeKha09}}}", shift right=0.5cm, Leftrightarrow, from=2-6, to=1-6]
  \arrow["{\text{ Thm.~\ref{thm:necessityFTGmodedependent} }}", shift left=0, Leftrightarrow, from=2-2, to=1-2]
  \arrow["{\text{Prop.~\ref{Prop:RobustMachin} }}", shift left=0.5, Rightarrow, from=1-6, to=1-2]
  \arrow["{\text{ Thm.~\ref{thm:HomogenenousCharachterization} }}", shift left=0, Leftrightarrow, from=2-2, to=2-1]
\end{tikzcd}
\caption{Schematic collection of the main results in the mode-independent case. The mode-dependent case  scheme is similar and can be simply obtained by the previous one, \emph{mutatis mutandis}. }  
\label{fig:dfsfds}
\end{figure*}

\vspace{-0.3cm}
\section{Conclusions}\label{sec:conclusion}
In this paper, we studied the stabilization of discrete-time switched linear systems in which the switching signal is treated as arbitrary and unmodifiable.  Our main result provides a complete characterization of stabilizability via necessary and sufficient conditions expressed in terms of linear matrix inequalities (LMIs) constraints, whose structure is based on a novel class of graphs, called \emph{feedback trees}, leading to  the design of \emph{piecewise linear} controllers.
 This result also clarifies the connection between our approach and previously established techniques relying on memory-dependent controllers in the literature.
Our results, their connection with existing literature and possible open questions are summarized in the scheme presented in Figure~\ref{fig:dfsfds}.
Future research directions include exploring further refinements of the proposed framework and the application of the proposed stabilization techniques for broader classes of hybrid and nonlinear control system.

\bibliography{biblio} 
    \bibliographystyle{elsarticle-harv}

\end{document}